\documentclass[12pt]{imsart}
\usepackage{xr}
\usepackage{amsthm,amsmath,amssymb}
\usepackage{parskip}
\usepackage[numbers]{natbib}
\usepackage[colorlinks]{hyperref}
\usepackage{hypernat}
\usepackage{marvosym}
\usepackage[hang,small,bf]{caption}
\usepackage{mathtools}

\newtheoremstyle{exampstyle}
{6pt} 
{6pt} 
{} 
{} 
{\bfseries} 
{.} 
{.5em} 
{} 

\theoremstyle{exampstyle}

\newtheorem{theorem}{Theorem}[section]

\newtheorem{lemma}[theorem]{Lemma}

\newtheorem{remark}{Remark}[section]

\newcommand{\E}{{\mathbb E}}

\newcommand{\eps}{{\epsilon}}

\newcommand{\R}{{\mathbb R}}

\renewcommand{\P}{{\mathbb P}}

\newcommand{\M}{{\mathcal{M}}}

\newcommand{\C}{{\mathcal{C}}}

\newcommand{\F}{{\cal F}}

\newcommand{\pp}{{\mathcal{P}}}

\newcommand{\ham}{{\Upsilon}}

\newcommand{\B}{\boldsymbol}
\newcommand{\bt}{\B \theta}
\newcommand{\beps}{\B \eps}
\newcommand{\by}{\B y}

\def\argmin{\mathop{\rm argmin}}

\def\qt#1{\qquad\text{#1}}

\makeatletter

\newcommand{\Rmnum}[1]{\expandafter\@slowromancap\romannumeral #1@}
\makeatother

\newcounter{rcnt}[section]

\begin{document}

\begin{frontmatter}
\title{On matrix estimation under monotonicity constraints}
\runtitle{Matrix estimation under monotonicity\;\;\;\;\;\;\;}

\begin{aug}
\author{\fnms{Sabyasachi} \snm{Chatterjee}\ead[label=e1]{sabyasachi@galton.uchicago.edu}},
\author{\fnms{Adityanand} \snm{Guntuboyina}\thanksref{t2}\ead[label=e2]{aditya@stat.berkeley.edu}} \and
\author{\fnms{Bodhisattva} \snm{Sen}\thanksref{t3}\ead[label=e3]{bodhi@stat.columbia.edu}}


\thankstext{t2}{Supported by NSF Grant DMS-1309356}
\thankstext{t3}{Supported by NSF Grant DMS-1150435}

\runauthor{Chatterjee, S., Guntuboyina, A. and Sen, B.}

\affiliation{University of Chicago, University of California at Berkeley and Columbia University}

\address{5734 S.~University Avenue \\
Chicago, IL 60637 \\
\printead{e1}
}

\address{423 Evans Hall\\
Berkeley, CA 94720 \\
\printead{e2}\\
\phantom{E-mail:\ }}

\address{1255 Amsterdam Avenue \\
New York, NY 10027\\
\printead{e3}
}
\end{aug}

\begin{abstract}
We consider the problem of estimating an unknown $n_1 \times n_2$
matrix $\bt^*$ from noisy observations under the constraint that
$\bt^*$ is nondecreasing in both rows and columns. We consider the
least squares estimator (LSE) in this setting and study its risk
properties. We show that the worst case risk of the LSE is $n^{-1/2}$,
up to multiplicative logarithmic factors, where $n = n_1  n_2$ and
that the LSE is minimax rate optimal (up to logarithmic factors). We
further prove that for some special $\bt^*$, the risk of the LSE could
be much smaller than $n^{-1/2}$; in fact, it could even be parametric,
i.e., $n^{-1}$ up to logarithmic factors. Such parametric rates occur
when the number of ``rectangular'' blocks of $\bt^*$ is bounded from above by a constant. We also derive an interesting adaptation property of the LSE which we term variable adaptation --- the LSE adapts to the ``intrinsic dimension'' of the problem and performs as well as the oracle estimator when estimating a matrix
that is constant along each row/column. Our proofs, which borrow ideas from
empirical process theory, approximation theory and convex geometry,  
are of independent interest.    
\end{abstract}

\end{frontmatter}

\section{Introduction}
This paper studies the problem of estimating an unknown $n_1 \times
n_2$ matrix $\bt^*$  under the constraint that $\bt^*$ is
nondecreasing in both rows and columns. In order to put this problem
and our results in proper context, consider first the problem of
estimating an unknown nondecreasing sequence under Gaussian
measurements. Specifically, consider the problem of estimating $\theta^* =
(\theta_1^*, \dots, \theta_n^*) \in \R^n$ from observations 
\begin{equation*}
  y_i = \theta_i^* + \epsilon_i \qt{for $i = 1, \dots, n$}
\end{equation*}
under the constraint that the unknown sequence $\theta^*$ satisfies $\theta_1^* \leq \dots \leq \theta_n^*$. Here the unobserved errors $\eps_1, \dots,\eps_n$ are i.i.d.~$N(0,\sigma^2)$ with $\sigma >0$ unknown. We refer to this as the vector isotonic estimation problem. This is a special case of
univariate isotonic regression and has a long history; see e.g.,~\citet{Brunk55}, \citet{AyerEtAl55}, and~\citet{vanEeden58}. The most commonly used estimator here is the least squares estimator (LSE) defined as  
\begin{equation}\label{vecp}
  \hat{\theta} := \argmin_{\theta \in \C_n} \sum_{i=1}^n \left(y_i - \theta_i
    \right)^2 \qt{where $\;\C_n := \left\{\theta \in \R^n: \theta_1 \leq
        \dots \leq \theta_n \right\}$}. 
\end{equation}
The behavior of $\hat{\theta}$ as an estimator of $\theta^*$ is most naturally studied in terms of the risk: 
\begin{equation*}
  R_{vec}(\theta^*, \hat{\theta}) := \frac{1}{n}\E_{\theta^*}
  \|\hat{\theta} - \theta^*\|^2
\end{equation*}
where $\|\cdot\|$ denotes the usual Euclidean norm. The subscript $vec$ is used to indicate that this denotes the risk in the vector estimation problem. This risk $R_{vec}(\theta^*, \hat{\theta})$ has been studied by a number of authors including  \citet{vdG90, vdG93},~\citet{Donoho91}, \citet{BM93}, \citet{Wang96}, \citet{MW00},
\citet{Zhang02} and \citet{us}. \citet{Zhang02}, among other things, showed the existence of a universal positive constant $C$ such that 
\begin{equation}\label{motw}
R_{vec}(\theta^*, \hat{\theta}) \leq C \left\{\left(\frac{\sigma^2
    \sqrt{D(\theta^*)}}{n} \right)^{2/3} +  \frac{\sigma^2 \log n}{n} \right\}.   
\end{equation}
with $D(\theta^*) := (\theta^*_n - \theta^*_1)^2.$ This result shows that the risk of $\hat{\theta}$ is no more than $n^{-2/3}$ (ignoring constant factors) provided $D(\theta^*)$ is bounded from above by a constant. It can be proved that $n^{-2/3}$ is the minimax rate of estimation in this problem (see e.g.,~\citet{Zhang02}). Throughout the paper, $C$ will denote a universal positive constant even though its exact value might change from place to place.  

A complementary upper bound on $R_{vec}(\theta^*, \hat{\theta})$ has
been proved recently by \citet{bellec2015sharp} who showed that 
\begin{equation}\label{usa}
  R_{vec}(\theta^*, \hat{\theta}) \leq  \inf_{\theta \in \C_n}
  \left(\frac{\|\theta^* - \theta\|^2}{n} + \frac{\sigma^2
    k(\theta)}{n} (\log e n) \right) 
\end{equation}
where $k(\theta)$ is the cardinality of the set $\{\theta_1,
\dots, \theta_n \}$. This result is an improvement of a previous
result by \citet{us} where inequality \eqref{usa} was
proved with an additional constant multiplicative factor. 

The two bounds \eqref{motw} and \eqref{usa} provide a holistic understanding of the global accuracy of the LSE $\hat{\theta}$ in vector isotonic estimation: its risk can never be larger than the minimax 
rate $(\sigma^2 \sqrt{D(\theta^*)}/n)^{2/3}$ while it can be the
parametric rate $\sigma^2/n$, up to logarithmic multiplicative
factors, if $\theta^*$ can be well approximated by $\theta \in \C_n$
with small $k(\theta)$. We refer to \eqref{motw} as the worst case
risk bound of the LSE and to \eqref{usa} as the adaptive risk bound
(adaptive because it states that the risk of the LSE is smaller than
the worst case rate for certain special $\theta^*$).  

The goal of this paper is to extend both these worst case and adaptive
risk bounds to the case of matrix isotonic estimation. Matrix isotonic
estimation refers to the problem of estimating an unknown matrix
$\bt^* = (\bt_{ij}^*) \in \R^{n_1 \times n_2}$ from observations   
\begin{equation}\label{eq:RegMdl}
\B{y}_{ij} = \bt_{ij}^* + \beps_{ij}, \qquad \mbox{for } i=1,\ldots,
n_1, \; j = 1,\ldots, n_2, 
\end{equation}
where $\bt^*$ is constrained to lie in
\begin{equation*}\label{eq:M}
  \M := \{ \bt \in \R^{n_1 \times n_2}: \bt_{ij} \leq \bt_{kl}
  \mbox{ whenever $i \leq k$ and $j \leq l$} \},  
\end{equation*}
and the random errors $\beps_{ij}$'s are i.i.d.~$N(0,\sigma^2)$,
with $\sigma^2>0$ unknown. We refer to any matrix in $\M$ as an isotonic
matrix. Throughout we let $n := n_1 n_2$ denote the product of $n_1$
and $n_2$. As a notational convention, throughout the paper, we denote
matrices in boldface and the $(i, j)$'th entry of a matrix $\B{A}$ will
simply be denoted by $\B{A}_{ij}$. 

Monotonicity restrictions on matrices are increasingly being used as a
key component of latent variable based models for the estimation of
matrices and graphs. Two such examples are: (1) the estimation of
graphons under monotonicity constraints (see~\citet{CA14}), and (2)
the nonparametric Bradley-Terry model (see~\citet{ChaUSVT15}). In both
of these examples, the unknown matrix satisfies monotonicity
constraints similar to the ones studied here. But the observation
model is more complicated because of the presence of latent
permutations. Nevertheless, we believe that studying the matrix
isotonic estimation problem described above is the first step towards
understanding the estimation properties in these more complicated
models.   

The matrix isotonic estimation problem also has a direct connection to
bivariate isotonic regression. Bivariate isotonic regression is the
problem of estimating a regression function $f:[0,1]^2 \to \R$ which
is known to be coordinate-wise nondecreasing (i.e., if $s_1 \le t_1$
and $s_2 \le t_2$, where 
  $(s_1,s_2),$ $(t_1, t_2) \in [0,1]^2$, then $f(s_1,s_2) \le 
f(t_1,t_2)$), from observations 
\begin{equation}\label{eq:BivarIsoReg} 
\B{y}_{ij} = f(i/n_1,j/n_2) + \beps_{ij},\quad \mbox{for } i = 1,\ldots,
n_1, \; j = 1,\ldots, n_2. 
\end{equation}
Identifying $f(i/n_1,j/n_2) \equiv \bt^*_{ij}$ we see
that~\eqref{eq:RegMdl} and~\eqref{eq:BivarIsoReg} are equivalent
problems. \eqref{eq:BivarIsoReg} is possibly the simplest example of a
multivariate shape constrained regression problem and arises quite
often in production planning and inventory control; see e.g., the
classical textbooks~\citet{BBBB72} and~\citet{RWD88} on this subject.    

Let us now introduce the LSE in matrix isotonic estimation. Let $\by =
(\B{y}_{ij})$ denote the matrix (of order $n_1 \times n_2$) of  
the observed response. The LSE, $\hat{\bt}$, is defined as the
minimizer of the squared Frobenius norm, $\|\by - \bt\|^2$, over $\bt
\in \M$, i.e.,  
\begin{equation}\label{eq:LSE}
\hat{\bt} := \argmin_{\bt \in \M} \sum_{i = 1}^{n_1} \sum_{j =
  1}^{n_2} (\B{y}_{ij} - \bt_{ij})^2.   
\end{equation}
Because $\M$ is a closed convex cone in $\R^{n_1 \times n_2}$ (which
is the space of all $n_1 \times n_2$ matrices), the LSE $\hat \bt$
exists uniquely. Further, it can be computed efficiently by an
iterative algorithm (see e.g.,~\citet{G70} and~\citet[Chapter 1]{RWD88}); this is in spite of the fact that it is defined as the solution of a quadratic program with $O(n^2)$ linear constraints.   

It is fair to say that not much is known about the behavior of
$\hat{\bt}$ as an estimator of $\bt^*$. The only result known in this
direction is the consistency of $\hat{\bt}$; see
e.g.,~\citet{HPW73},~\citet{M77} and~\citet{RW75}. In this paper we
study the risk of $\hat{\bt}$ as an estimator of $\bt^*$, defined as  
\begin{equation*}
  R(\bt^*, \hat{\bt}) := \E_{\bt^*} \ell^2 \big(\bt^*, \hat{\bt}\big)
  \qt{where $\;\ell^2(\bt^*, \bt) := \frac{1}{n} 
    \sum_{i=1}^{n_1} \sum_{j=1}^{n_2} \left(\bt^*_{ij} - \bt_{ij} 
    \right)^2$}. 
\end{equation*}
Here $\E_{\bt^*}$ denotes the expectation taken with respect to $\by$
having the distribution given by~\eqref{eq:RegMdl}. Also, throughout
the paper, we take $n = n_1 n_2$ and each of $n_1$ and $n_2$ to be
strictly larger than one. We similarly define the risk $R(\bt^*,
\tilde{\bt})$ for any other estimator $\tilde{\bt}$ of $\bt^*$.      

To the best of our knowledge, nothing is known in the literature about
the risk $R(\bt^*, \hat{\bt})$. The goal of this paper is to prove
analogues of the inequalities \eqref{motw} and \eqref{usa} for
$R(\bt^*, \hat{\bt})$. The first result of this paper, Theorem
\ref{kp1}, is the analogue of \eqref{motw} for matrix isotonic 
estimation. Specifically, we prove in Theorem \ref{kp1} that 
\begin{equation}\label{wwc}
R(\bt^*, \hat{\bt}) \leq C \left(\sqrt{\frac{\sigma^2 D(\bt^*)}{n}}
  (\log n)^4 + \frac{\sigma^2}{n} (\log n)^8 \right)   
\end{equation}
for a universal positive constant $C$ where $D(\bt^*) :=
(\bt_{n_1n_2}^* - \bt_{11}^*)^2$.   

Our second result proves that the minimax risk in this problem is
bounded from below by $(\sigma^2 D(\bt^*)/n)^{1/2}$, up to constant
multiplicative factors. Specifically, we prove in Theorem \ref{lobo}
that  
\begin{equation}\label{minu}
  \inf_{\tilde{\bt}} \sup_{\bt \in \M: D(\bt) \leq D}  R(\bt,
  \tilde{\bt}) \geq \sqrt{\frac{\sigma^2 D}{192 n}} 
\end{equation}
under some conditions on $n_1$ and $n_2$ (see Theorem \ref{lobo} for
the precise statement). The above infimum is taken over all estimators
$\tilde{\bt}$ of $\bt$. Combined with \eqref{wwc}, this proves that
$\hat{\bt}$ is minimax, up to logarithmic multiplicative
factors. Therefore, inequality \eqref{wwc} is the correct analogue of
\eqref{motw} for matrix isotonic estimation.  

Next we describe our analogue of inequality \eqref{usa} for matrix
isotonic estimation. The situation here is more subtle compared to the
vector case. The most natural analogue of \eqref{usa} in the 
matrix case is an inequality of the form: 
\begin{equation}\label{incad}
  R(\bt^*, \hat{\bt}) \leq  \inf_{\bt \in \M}
  \left(\frac{\|\bt^* - \bt\|^2}{n} + \frac{\sigma^2 c(\bt) p(\log
      n)}{n} \right) 
\end{equation}
where $p(\cdot)$ is some polynomial and $c(\bt)$ denotes the
cardinality of the set $\{\bt_{ij}: 1 \leq i \leq n_1, 1 \leq j \leq
n_2 \}$ and $\|\cdot\|$ refers to the Frobenius norm. Unfortunately it
turns out that this inequality cannot be true for every $\bt^* \in \M$
because it contradicts the minimax lower bound \eqref{minu}. The
argument for this is provided at the beginning of Section~\ref{adabo}.       

The fact that inequality \eqref{incad} is false means that the LSE $\hat{\bt}$ does not adapt to every $\bt^* \in \M$ with small $c(\bt^*)$. It turns out that inequality \eqref{incad} can be proved for every $\bt^* \in \M$ if the quantity $c(\bt)$ is replaced by a larger quantity. This quantity will be denoted by $k(\bt)$ (because it is the right analogue of $k(\theta)$ for the matrix case) and it is defined next after introducing some notation. 

A subset $A$ of $\{1, \dots, n_1\} \times \{1, \dots, n_2\}$ is called
a rectangle if $A = \{(i, j) : k_1 \leq i \leq l_1, k_2 \leq j \leq
l_2 \}$ for some $1 \leq k_1 \leq l_1 \leq n_1$ and $1 \leq k_2 \leq
l_2 \leq n_2$. A rectangular partition of $\{1, \dots, n_1\} \times
\{1, \dots, n_2\}$ is a collection of rectangles $\pi = (A_1,  
\dots, A_k)$ which are disjoint and whose union is $\{1, \dots, n_1\}
\times \{1, \dots, n_2\}$. The cardinality of such a partition,
$|\pi|$, is the number of rectangles in the partition. The collection
of all rectangular partitions of $\{1, \dots, n_1\} \times \{1, \dots, n_2\}$ will be denoted by $\pp$. For $\bt \in \M$ and $\pi = (A_1, \dots, A_k) \in \pp$, we say that $\bt$ is constant on $\pi$ if
$\{\bt_{ij} : (i,j) \in A_l\}$ is a singleton for each $l = 1, \dots,
k$. We are now ready to define $k(\bt)$ for $\bt \in \M$. It is
defined as the ``number of rectangular blocks'' of $\bt$, i.e., the smallest integer $k$ for which there exists a partition $\pi \in \pp$ with $|\pi| = k$ such that $\bt$ is constant on $\pi$. It is trivial to see that $k(\bt) \geq c(\bt)$ for every $\bt \in \M$. As a simple illustration, for $\bt = \mathbf{1}\{i > 1, j > 1\},$ we have $c(\bt) = 2$ and $k(\bt) = 3.$   

Inequality \eqref{incad} becomes true for all $\bt^* \in \M$ if
$c(\bt)$ is replaced by $k(\bt)$. This is our
adaptive risk bound for matrix isotonic estimation, proved in Theorem
\ref{kp2}: 
\begin{equation}\label{tama}
  R(\bt^*, \hat{\bt}) \leq  \inf_{\bt \in \M}  \left(\frac{\|\bt^* -
    \bt\|^2}{n} + \frac{C \sigma^2 k(\bt)}{n} (\log n)^8 \right). 
\end{equation}
where $C$ is a universal positive constant. As a consequence of this
inequality, we obtain that the risk of the LSE converges to zero at
the parametric rate $\sigma^2/n$, up to logarithmic multiplicative
factors, provided $k(\bt^*)$ is bounded from above by a constant.    

We also establish a property of the LSE that we term variable
adaptation. Let $\C_{n_1} := \left\{\theta \in 
  \R^{n_1}: \theta_1 \leq \dots \le \theta_{n_1} \right\}$. Suppose
$\bt^* = (\bt^*_{ij}) \in \M$  has the property that $\bt^*_{ij}$ only
depends on $i$, i.e., there exists $\theta^* \in \C_{n_1}$ such that
$\bt^*_{ij} = \theta^*_i$ for every $i$ and $j$. If we knew this fact
about $\bt^*$, then the most natural way of estimating it would be to
perform vector isotonic estimation based on the row-averages $\bar{y}
:= \left(\bar{y}_{1}, \dots,   \bar{y}_{n_1} \right)$, where
$\bar{y}_{i} := \sum_{j=1}^{n_2} \by_{ij}/n_2$, resulting in an
estimator $\breve{\bt}$ of $\bt^*$. Using the vector isotonic risk
bounds \eqref{motw} and \eqref{usa}, it is easy to see then that the
risk of $\breve{\bt}$ has the following pair of bounds:  
\begin{equation}\label{eq:VarAdap1}
  R(\bt^*, \breve{\bt}) \leq C \left\{ \left(\frac{\sigma^2
        \sqrt{D(\bt^*)}}{n} \right)^{2/3} + \frac{\sigma^2 \log
    n_1}{n} \right\} 
\end{equation}
and 
\begin{equation}\label{eq:VarAdap2}
  R(\bt^*, \breve{\bt}) \leq  \inf_{\theta \in \C_{n_1}}
  \left(\frac{\|\theta^* - \theta \|^2}{n_1} + \frac{\sigma^2
    k(\theta)}{n} \log n_1 \right). 
\end{equation}
Note that the construction of $\breve{\bt}$ requires the knowledge
that all rows of $\bt^*$ are constant. As a consequence of the
adaptive risk bound \eqref{tama}, we shall show in Theorem
\ref{varadap} that the matrix isotonic LSE $\hat{\bt}$ achieves the same risk bounds as $\breve{\bt}$, up to additional logarithmic factors. This is remarkable because $\hat{\bt}$ uses no special knowledge on $\bt^*$; it automatically adapts to the additional structure present in $\bt^*$. 

Note that in the connection between matrix isotonic estimation and
bivariate isotonic regression, the assumption that $\bt^*_{ij} =
f(i/n_1, j/n_2)$ does not depend on $j$ is equivalent to assuming that
$f$ does not depend on its second variable. Thus, when estimating a
bivariate isotonic regression function that only depends on one
variable, the LSE automatically adapts and we get risk bounds that
correspond to estimating a monotone function of one variable. This is
the reason why we refer to this phenomenon as variable adaptation. To
the best of our knowledge, such a result on automatic variable
adaptation in multivariate nonparametric regression is very rare 
--- most nonparametric regression techniques
(e.g., kernel smoothing, splines) do not exhibit such automatic adaptation properties. 

The proof techniques employed in this paper are quite different from the case of vector isotonic estimation. In the vector problem
\eqref{vecp}, the LSE has the closed form expression (see e.g.,
\citet[Chapter 1]{RWD88}): 
\begin{equation}\label{rwd}
  \hat{\theta}_i := \min_{v \geq i} \max_{u \leq i} \frac{1}{v-u+1}
    \sum_{i=u}^v y_i.
\end{equation}
This expression, along with some martingale maximal inequalities, are crucially used for the proofs of inequalities \eqref{motw} and
\eqref{usa}; see e.g.,~\citet{Zhang02} and~\citet{us}. The LSE \eqref{eq:LSE} in the matrix estimation problem also has a closed form expression similar to \eqref{rwd}: 
\begin{equation}\label{setm}
  \hat{\bt}_{ij} = \min_{L \in \mathcal{L} : (i, j) \in L}  \max_{U
    \in \mathcal{U} : (i, j) \in U} \bar{y}_{L \cap U} 
\end{equation}
where $\mathcal{L}$ and $\mathcal{U}$ denote the collections of all lower sets and upper sets respectively and $\bar{y}_A$ is the average of $\{\by_{ij}: (i, j) \in A\}$; see \citet[Chapter 1]{RWD88} for the definitions of upper and lower sets and for a proof of~\eqref{setm}. This unfortunately is a much more complicated expression to directly work with compared to \eqref{rwd}. It is not clear to us if simple martingale techniques can be used in conjunction with the expression \eqref{setm} to prove risk bounds for the LSE. 

We therefore abandon the direct approach based on the expression
\eqref{setm} and instead resort to general techniques for LSEs in
order to prove our results. Specifically, we use the standard
empirical process based approach to prove the worst case bound
\eqref{wwc}. This approach relies on metric entropy calculations of
the space of isotonic matrices. Metric entropy results for classes of
isotonic matrices can be derived from those of bivariate
coordinate-wise nondecreasing functions. However existing metric
entropy results for classes of bivariate nondecreasing functions (as
in \citet{GW07}) require the functions to be uniformly
bounded. Because of this reason, these results are not directly
applicable to our setting. We suitably extend these results
in order to allow for the lack of a uniform bound. On the other hand,
for the adaptive risk bound~\eqref{tama}, we use connections between
the risk of LSEs and size measures of tangent cones. Thus, our proofs
borrow ideas from empirical process theory, approximation theory and
convex geometry and are  of independent interest.   

The rest of the paper is organized as follows. Our results are
described in Section \ref{mre}: Subsection \ref{wwcF} deals with the
worst case risk bounds while Subsection \ref{adabo} focuses on the adaptive
bounds. In Section \ref{pwc}, we provide the necessary background on
the general theory of the LSEs, prove our main metric entropy results
and present the proof of our main worst case upper bound. In Section
\ref{2.3}, we discuss connections between risk of LSEs and appropriate
size measures of tangent cones, and also present the proof of our
adaptive risk bounds. Additional discussion is provided in Section
\ref{cuss}. We have also included an Appendix which contains the
proofs of certain auxiliary technical results used in the paper. 


\section{Main Results}\label{mre}
In this section we give risk bounds on the performance of the isotonic LSE $\hat{\bt}$, defined in~\eqref{eq:LSE}. We start with a generalization of~\eqref{motw} and then proceed to exhibit the adaptive risk behavior of $\hat{\bt}$.  We end this section with a result on the variable adaptation property of the LSE which shows that $\hat{\bt}$ automatically adapts to the intrinsic dimension of the problem.

\subsection{Worst case risk bounds}\label{wwcF}
Our first main result establishes inequality \eqref{wwc} which gives an upper bound on the worst case risk of the matrix isotonic LSE $\hat{\bt}$. We will actually prove a slightly stronger bound than
that given by inequality \eqref{wwc}. We first need some notation. We
define the {\it variance} of a matrix $\bt$ as   
\begin{equation}\label{vdef}
V(\bt) := \frac{1}{n} \sum_{i=1}^{n_1} \sum_{j=1}^{n_2} (\bt_{ij}
- \overline{\theta})^2, 
\end{equation}
where $\overline{\theta} = \sum_{i = 1}^{n_1} \sum_{j =1}^{n_2}
\bt_{ij}/n$ is the mean of the entries of $\bt$. Note that
$V(\bt) \leq D(\bt)$ for every $\bt \in \M$. We  also denote the
set $\{1,\dots,l\}$ by $[l]$ for positive integers $l$.     

The following theorem, proved in Section~\ref{compe}, gives an upper bound on the risk $R(\bt^*, \hat{\bt})$ in terms of the quantity $V(\bt^*)$. Because $V(\bt^*)
\leq D(\bt^*)$, the conclusion of the theorem is stronger than
inequality \eqref{wwc}.
\begin{theorem}\label{kp1}
There exists a universal positive constant $C$ such that for every
$n_1, n_2 > 1$ with $n = n_1 n_2$ and $\bt^* \in \M$, 
\begin{equation*}\label{eq:kp1}
  R(\bt^*, \hat{\bt}) \leq C \left(\frac{\sigma^2}{n} (\log n)^8 +
    \sqrt{\frac{\sigma^2 V(\bt^*)}{n}} (\log n)^4 \right).      
\end{equation*}
\end{theorem}

Ignoring constants and logarithmic factors, Theorem \ref{kp1} states
that the risk of the LSE at $\bt^*$ converges to zero at the rate
$n^{-1/2}$ as long as $V(\bt^*)$ is bounded away from zero. In the
next result, proved in Appendix~\ref{App:A4}, we argue that $n^{-1/2}$ is also a minimax lower bound
in this problem. This implies that the rate $n^{-1/2}$ cannot be
improved by any other estimator uniformly over the class $\{\bt^*:
V(\bt^*) \leq V\}$ for every  constant $V$. 

\begin{theorem}\label{lobo}
For every positive real number $D$, 
\begin{equation}
  \label{lobo.eq}
  \inf_{\tilde{\bt}} \sup_{\bt \in \M: D(\bt) \leq D}
  R(\bt, \tilde{\bt}) \geq \sqrt{\frac{\sigma^2 D}{192n}}  
\end{equation}
where the infimum is over all estimators $\tilde{\bt}$ of $\bt$, provided the integers $n_1 \geq 1, n_2 \geq 1$ with $n = n_1 n_2$ satisfy $n \geq 9 \sigma^2/D$ and  
\begin{equation}\label{lobo.con}
 \min \left(\frac{n_1^3}{n_2} , \frac{n_2^3}{n_1} \right) \geq
 \frac{D}{9 \sigma^2}.  
\end{equation}
\end{theorem}

\begin{remark}
  The condition \eqref{lobo.con} is necessary to ensure that neither
  $n_1$ or $n_2$ are too small. Indeed, the inequality
  \eqref{lobo.eq} is not true when, for example, $n_1 = 1,n_2 = n$ because in
  this case the problem reduces to vector isotonic estimation where
  the minimax risk is of the order $n^{-2/3} < 
  n^{-1/2}$. When $n_1 = n_2 = \sqrt{n}$, the inequality
  \eqref{lobo.con}  is equivalent to $n \geq D/(9 \sigma^2)$ which is
  satisfied for all large $n$.   
\end{remark}

\begin{remark}
  Recall the quantity $V(\bt)$ defined in \eqref{vdef}. Because
  $V(\bt) \leq D(\bt)$, it follows that $\{\bt: D(\bt) \leq D\}
  \subseteq \{\bt: V(\bt) \leq D\}$. Therefore the bound
  \eqref{lobo.eq} also holds if $\{\bt: D(\bt) \leq D\}$ is replaced
  by the larger set $\{\bt: V(\bt) \leq D\}$. 
\end{remark}

In addition to proving that the LSE is minimax optimal up to
logarithmic factors, another interesting aspect of Theorem \ref{kp1}
is that when $V(\bt^*) = 0$, the upper bound on $R(\bt^*, \hat{\bt})$
becomes the parametric rate $\sigma^2/n$ up to a logarithmic
factor. This rate is faster than the worst case rate $n^{-1/2}$. Thus
the LSE adapts to $\bt^* \in \{\bt: V(\bt) = 0\}$. A more detailed
description of the adaptation properties of the LSE is provided in the
next theorem. 

\subsection{Adaptive risk bounds}\label{adabo}
The adaptation properties of the matrix isotonic LSE are more subtle compared to the vector case. In the latter case, adaptation of the LSE is described by inequality \eqref{usa}. The most  natural analogue of~\eqref{usa} in the  matrix case is an inequality of the form~\eqref{incad}.  
Unfortunately it
turns out that this inequality cannot be true for every $\bt^* \in \M$
because it contradicts the minimax lower bound 
proved in Theorem \ref{lobo}. The reason for this is the
following. Fix $\bt^* = (\bt^*_{ij}) \in \M$ with $D := D(\bt^*) =
(\bt^*_{n_1n_2} - \bt^*_{11})^2 > 0$.  Now fix $c \geq 1$ and define
$\bt = (\bt_{ij})$ by   
\begin{equation*}
  \bt_{ij} :=  \theta^*_{11} + \frac{\sqrt{D}}{c} \left \lfloor
  \frac{c(\bt^*_{ij} - \bt^*_{11})}{\sqrt{D}} \right \rfloor. 
\end{equation*}
It is easy to see that $\bt \in \M$ (because $\bt_{ij}$ is a
nondecreasing function of $\bt_{ij}^*$). Also for every $i, j$, we
have $\bt^*_{ij} - \sqrt{D}/c \leq \bt_{ij} \leq \bt^*_{ij}$ which
implies that $\|\bt - \bt^*\|^2 \leq n D/c^2$. Finally 
$c(\bt) \leq (c+1)$. Therefore if inequality \eqref{incad} were true
for every $\bt^*$, we would obtain  
\begin{equation*}
  R(\bt^*, \hat{\bt}) \leq p(\log n) \inf_{c \geq 1}
  \left(\frac{D}{c^2} + \frac{\sigma^2 (c+1)}{n} \right). 
\end{equation*}
Choosing $c = \lfloor (n D/\sigma^2)^{1/3} \rfloor$, we would obtain
that $R(\bt^*, \hat{\bt})$ converges to zero at the $n^{-2/3}$
rate. This obviously contradicts the minimax lower bound proved in
Theorem \ref{lobo}. Therefore, one cannot hope to prove an inequality 
of the form \eqref{incad} for every $\bt^* \in \M$.

The fact that inequality \eqref{incad} is false means that the LSE  $\hat{\bt}$ does not adapt to every $\bt^* \in \M$ with small $c(\bt^*)$. However, inequality \eqref{incad} can be proved for every $\bt^* \in \M$ if the quantity $c(\bt)$ is replaced by the larger quantity $k(\bt)$ --- the number of rectangular blocks --- as defined in the Introduction.
We are now ready to state our main adaptive risk bound for the matrix LSE; see Section~\ref{Sec4.2} for its proof.  
\begin{theorem}\label{kp2}
There exists a universal constant $C>0$ such that for every $\bt^* \in
\M$ we have  
  \begin{equation}\label{eq:kp2}
     R(\bt^*, \hat{\bt}) \leq  \inf_{\bt \in \M}  \left\{\frac{
         \|\bt^* - \bt\|^2}{n} +  \frac{C k(\bt) \sigma^2}{n}
       (\log n)^8 \right\}.   
  \end{equation}
\end{theorem}

\begin{remark}
Note that $1 \leq k(\bt) \leq n$ for all $\bt \in \M$.   
There exist $\bt \in \M$ for
  which $c(\bt) = k(\bt)$. These are elements $\bt \in \M$ whose level
  sets (level sets of $\bt$ are non-empty sets of the form $\{(i, j) :
  \bt_{ij} = a\}$ for some real number $a$) are all rectangular. 
\end{remark}

\begin{remark} 
A simple consequence of Theorem \ref{kp2} is that $R(\bt^*,
\hat{\bt})$ is bounded by the parametric rate (up to logarithmic
factors) when $k(\bt^*)$ is bounded from above by a constant. To see
this, simply note that we can take $\bt = \bt^*$ in \eqref{eq:kp2}
to obtain   
  \begin{equation*}
    R(\bt^*, \hat{\bt}) \leq C (\log n)^8 \frac{k(\bt^*)
      \sigma^2}{n}. 
  \end{equation*}
The right hand side above is just the parametric rate $\sigma^2/n$  up to logarithmic factors provided $k(\bt^*)$ is bounded by a constant (or a logarithmic factor of $n$). 
\end{remark}

\begin{remark}\label{anre}
 Inequality~\eqref{eq:kp2} sometimes gives near parametric bounds for
  $R(\bt^*, \hat{\bt})$ even when $k(\bt^*) = n$. This happens when
  $\bt^*$ is well approximated by some $\bt \in \M$ with small
  $k(\bt)$. An example of this is given below: Assume,
for simplicity, that $n_1 = n_2 = \sqrt{n} = 2^k$ for some positive
integer $k$. Define $\bt^* \in \R^{n_1 \times n_2}$ by  
\begin{equation*}
  \bt^*_{ij} = -(2^{-i} + 2^{-j})  \qt{for $1 \leq i, j \leq n_1$}.  
\end{equation*}
It should then be clear that $\bt^* \in \M$ and $k(\bt^*) =
n$. Also, let us define $\bt \in \M$ by 
\begin{equation*}
  \bt_{ij} = -(2^{-(i \wedge k)} + 2^{-(j \wedge k)})  \qt{for $1 \leq i, j
    \leq n_1$},
\end{equation*}
where $a \wedge b := \min(a, b)$. Observe that $k(\bt) \leq (k+1)^2 \leq C \log n$. Further  
\begin{equation*}
  \frac{1}{n} \|\bt - \bt^* \|^2 \leq \max_{(i, j)} (\bt_{ij} -
  \bt^*_{ij} )^2 \leq 2 \left(2^{-2k} + 2^{-2k} \right) =
  \frac{4}{n}. 
\end{equation*}
Theorem \ref{kp2} therefore gives
\begin{equation*}
  R(\bt^*, \hat{\bt}) \leq C \left\{\frac{1}{n} + \frac{\sigma^2}{n}
    (\log n)^9 \right\}. 
\end{equation*}
This is the parametric bound up to logarithmic factors in $n$. 
\end{remark}

\subsection{Variable adaptation}
In this sub-section we describe a very interesting property of the LSE which shows that $\hat{\bt}$ adapts to the intrinsic dimension of the problem.  Suppose that $\bt^* \in \M$ is such that its value does not depend on the columns, i.e., there exists $\theta^* \in \C_{n_1}$ (recall that $\C_{n_1} = \left\{\theta^* \in 
  \R^{n_1}: \theta_1^* \leq \dots \leq \theta_{n_1}^* \right\}$) such that $\bt^*_{ij} = \theta^*_i$ for every $i$ and $j$. Note that in connection to bivariate isotonic regression, the assumption that $\bt^*_{ij} :=  f(i/n_1, j/n_2)$ does not depend on $j$ is equivalent to assuming that $f$ does not depend on its second variable. If we knew this fact about $\bt^*$, then the most natural way of estimating it would be to perform vector isotonic estimation based on the row-averages $\bar{y} := \left(\bar{y}_{1}, \dots,   \bar{y}_{n_1} \right)$, where $\bar{y}_{i} := \sum_{j=1}^{n_2} \by_{ij}/n_2$, resulting in an estimator $\breve{\bt}$ of $\bt^*$. This oracle estimator has risk bounds given in~\eqref{eq:VarAdap1} and~\eqref{eq:VarAdap2}. 
  
  The following theorem, proved in Section~\ref{Sec4.3}, shows that the matrix isotonic LSE $\hat{\bt}$
  achieves the same risk bounds as $\breve{\bt}$, up to additional
  multiplicative logarithmic factors. This is remarkable because
  $\hat{\bt}$ uses no special knowledge on $\bt^*$; it automatically
  adapts to the additional structure present in $\bt^*$. Thus, when
  estimating a bivariate isotonic regression function that only
  depends on one variable, the LSE automatically adapts and we get
  risk bounds that correspond to estimating a monotone function in one
  variable. As mentioned in the Introduction such a result on
  automatic variable adaptation in multivariate nonparametric
  regression is very rare. 
\begin{theorem}{\label{varadap}}
Suppose $\bt^* = (\bt^*_{ij}) \in \M$ and $\theta^* = (\theta_i^*) \in \C_{n_1}$ are such that $\bt^*_{ij} = \theta^*_i$ for all $1 \leq i \leq n_1$ and $1 \leq j \leq n_2$. Then the following pair of
inequalities hold for a universal positive constant $C$: 
\begin{equation}\label{vara1}
R(\bt^*, \hat{\bt}) \leq  \inf_{\theta \in \C_{n_1}} 
        \left\{ \frac{ \|\theta^* - \theta \|^2}{n_1}  + \frac{C
            k(\theta) \sigma^2}{n} (\log n)^8\right\} 
\end{equation}
and 
\begin{equation}\label{vara2}
  R(\bt^*, \hat{\bt}) \leq C (\log n)^8 \left(\frac{\sigma^2
      \sqrt{D(\bt^*)}}{n} \right)^{2/3} \quad \textrm{ provided }n D(\bt^*) \geq
    2 \sigma^2. 
\end{equation}
\end{theorem}

\section{General theory of LSEs, Metric Entropy Calculations and the
  proof of Theorem \ref{kp1}}  \label{pwc}

This section is mainly devoted to the proof of Theorem
\ref{kp1}. The general theory of LSEs under convex constraints is
crucially used to prove Theorem \ref{kp1}. Parts of this general
theory that are relevant to the proof of Theorem \ref{kp1} are
recalled in the next subsection.  Essentially, this general theory
reduces the problem of bounding $R(\bt^*, \hat{\bt})$ to certain
metric entropy calculations of classes of isotonic matrices. In
Subsection \ref{mms}, we prove such results by extending appropriately
existing metric entropy results for bivariate coordinate-wise
nondecreasing functions due to \citet{GW07}. Finally, in Subsection
\ref{compe}, we complete the proof of Theorem \ref{kp1} by combining
the metric entropy results with general results on LSEs.     

\subsection{General Theory of LSEs}
The following result due to \citet[Corollary 1.2]{Chat14} is a key
technical tool for the proof of Theorem \ref{kp1}. It reduces
the problem of bounding $R(\bt^*, \hat{\bt})$ to controlling the
maximizer of an appropriate Gaussian process.  

\begin{theorem}[Chatterjee]\label{chatthm}
Fix $\bt^* \in \M$. Let us define the function $f_{\bt^*}: \R_{+}
\rightarrow \R$ as 
\begin{equation}\label{eq:chat}
  f_{\bt^*}(t)  := \E \left(\sup_{\bt \in \M: \|\bt^* - \bt\| \leq t}
     \sum_{i=1}^{n_1} \sum_{j=1}^{n_2}  \beps_{ij} 
    \left(\bt_{ij} - \bt^*_{ij} \right) \right)
  - \frac{t^2}{2}.
\end{equation}
Let $t_{\bt^*}$ be the point in $[0, \infty)$ where $t \mapsto
f_{\bt^*}(t)$ attains its maximum (existence and uniqueness of
$t_{\bt^*}$ are proved in~\citet[Theorem 1.1]{Chat14}). Then there
exists a  universal positive constant $C$ such that   
\begin{equation}\label{cim}
R(\bt^*, \hat{\bt}) \leq \frac{C}{n} \max \left(t_{\bt^*}^2, \sigma^2 \right).  
\end{equation}
\end{theorem}
The above theorem reduces the problem of bounding $R(\bt^*,
\hat{\bt})$ to that of bounding $t_{\bt^*}$. For
this latter problem,~\citet[Proposition 1.3]{Chat14} observed that 
\begin{equation*}\label{chat2}
  t_{\bt^*} \leq t^{**} \qt{whenever $t^{**} > 0$ and
    $f_{\bt^*}(t^{**}) \leq 0$}.  
\end{equation*}
In order to bound $t_{\bt^*}$, one therefore seeks $t^{**} > 0$ such
that $f_{\bt^*}(t^{**}) \leq 0$. This now requires a bound on the
expected supremum of the Gaussian process in the definition of
$f_{\bt^*}(t)$ in  \eqref{eq:chat}. 

It will be convenient below to have the following notation. For
$n_1 \times n_2$ matrices $\B{M}, \B{N} \in \R^{n_1 \times n_2}$, let
$\|\B{M} - \B{N} \|$ denote the Frobenius distance between $\B{M}$ and
$\B{N}$ defined by 
\begin{equation*}
  \|\B{M} - \B{N}\|^2 :=  \sum_{i=1}^{n_1} \sum_{j=1}^{n_2}
  \left(\B{M}_{ij} - \B{N}_{ij} \right)^2 . 
\end{equation*}
For a subset $\F \subseteq \R^{n_1 \times n_2}$ and $\epsilon > 0$,
let $N(\epsilon, \F)$ denote the $\epsilon$-covering number of $\F$
under the Frobenius metric $\|\cdot\|$ (i.e., $N(\epsilon, \F)$ is the
minimum number of balls of radius $\eps$ required to cover
$\F$). Also, for each $\bt^* \in \M$  and $t > 0$, let   
\begin{equation}\label{bn}
  B(\bt^*, t) := \left\{\bt \in \M: \|\bt - \bt^*\| \leq t \right\}
\end{equation}
denote the ball of radius $t$ around $\bt^*$. Observe that the
supremum in the definition of \eqref{eq:chat} is over all $\bt \in
B(\bt^*, t)$. Finally let  
\begin{equation*}
  \langle \beps, \bt - \bt^* \rangle := \sum_{i=1}^{n_1} \sum_{j=1}^{n_2}
  \beps_{ij} \left(\bt_{ij} - \bt^*_{ij} \right). 
\end{equation*}
The following chaining result gives an upper bound on the expected suprema of the above Gaussian process (see e.g., \citet{VandegeerBook}); see~\cite{CGS15} for a proof. 
\begin{theorem}[Chaining]\label{dudthm}
For every $\bt^* \in \M$ and $t > 0$,
{\begin{equation*}\label{gaup}
\E \left[ \sup_{\bt \in B(\bt^*, t)} \left<\beps, \bt - \bt^* \right> \right] \leq  \sigma \inf_{0 < \delta
      \leq 2t} \left\{12 \int_{\delta}^{2t}
      \sqrt{\log N(\epsilon, B(\bt^*, t))} \;
     d\epsilon + 4 \delta \sqrt{n} \right\} .  
\end{equation*}}
\end{theorem}


The general results outlined here essentially reduce the problem of
bounding $R(\bt^*, \hat{\bt})$ to controlling the metric entropy of
subsets of $\M$ of the form $B(\bt^*, t)$. Such a metric entropy bound
is proved in the next subsection. This is the key technical component
in the proof of Theorem \ref{kp1}. 

\subsection{Main metric entropy result}\label{mms}
Let $\B{0}$ denote the $n_1 \times n_2$ matrix all of whose entries
are equal to 0. According to the notation \eqref{bn}, we have
\begin{equation}\label{bo}
  B(\B{0}, 1) = \left\{\bt \in \M : \|\bt - \B{0}\| \leq 1 \right\} =
  \left\{\bt \in \M : \sum_{i=1}^{n_1} \sum_{j=1}^{n_2}
    \bt_{ij}^2 \leq 1 \right\}.
\end{equation}
The next theorem gives an upper bound on the $\epsilon$-covering
number of $B(\B{0}, 1)$ (all covering numbers will be with respect to 
the Frobenius metric $\|\cdot\|$). It will be crucially used in our
proof of Theorem \ref{kp1}.   

\begin{theorem}{\label{sdb}}
There exists a universal positive constant $C$ such that the following
inequality holds for every $\epsilon > 0$ and integers $n_1, n_2 > 1$:  
\begin{equation}\label{sdb.eq}
  \log N(\epsilon, B(\B{0}, 1)) \leq C \frac{(\log n_1)^2 (\log 
    n_2)^2}{\epsilon^2} \left[\log \frac{4 \sqrt{\log n_1 \log
        n_2}}{\epsilon} \right]^2. 
\end{equation}
Moreover for every $0 < \delta \leq 1$, 
{\small \begin{equation}\label{cor1}
  \int_{\delta}^1 \sqrt{\log N(\epsilon, B(\B{0}, 1))} d \epsilon \leq
  \frac{\sqrt{C}}{2} (\log n_1) (\log n_2)  \left(\log \frac{4
      \sqrt{\log n_1 \log n_2}}{\delta} \right)^2. 
\end{equation}}
\end{theorem}

There is a close connection between metric entropy results for
isotonic matrices and those for bivariate coordinate-wise
nondecreasing functions. Indeed, for every isotonic matrix $\bt$, we
can associate a bivariate coordinate-wise nondecreasing function
$\phi_{\bt} : [0, 1]^2 \rightarrow \R$ via 
\begin{equation*}
  \phi_{\bt}(x_1, x_2) := \min \left\{\bt_{ij} : n_1 x_1 \leq i \leq 
    n_1, n_2 x_2 \leq j \leq n_2 \right\} 
\end{equation*}
for all $(x_1, x_2) \in [0, 1]^2$. It can then be directly verified
that  
\begin{equation*}
\|\bt - \B{\nu}\|^2  = n \int_0^1 \int_{0}^1
(\phi_{\bt}(x_1,x_2) -  \phi_{\B{\nu}}(x_1,x_2))^2 dx_1dx_2  
\end{equation*}
for every pair $\bt,\B{\nu}$ of isotonic matrices. This
means that metric entropy results for classes of isotonic matrices can
be derived from those of bivariate nondecreasing functions. However
existing metric entropy results for classes of bivariate nondecreasing
functions (see \citet{GW07}) require the functions to be uniformly
bounded. If the average constraint in the definition \eqref{bo} of
$B(\B{0}, 1)$ is replaced by a supremum constraint i.e., if one
considers the smaller set $B_{\infty} (\B{0}, n^{-1/2}) := \{\bt \in
\M : \sup_{1 \leq i \leq   n_1, 1 \leq j \leq n_2}   \left| \bt_{ij}
\right| \leq n^{-1/2} \}$, then the metric entropy of
$B_{\infty}(\B{0}, n^{-1/2})$ can be easily controlled via the results
of \citet{GW07}. This is the content of the following lemma where we
actually consider the classes     
  \begin{equation*}
    B_{\infty} (\B{0}, t) := \left\{\bt \in \M : \sup_{1 \leq i \leq
        n_1, 1 \leq j \leq n_2} \left| \bt_{ij} \right| \leq t
    \right\}  
  \end{equation*}
  for general $t > 0$. 
\begin{lemma}\label{bmi}
There exists a universal positive constant $C$ such that 
  \begin{equation*}
    \log N(\epsilon,B_{\infty}(\B{0}, t)) \leq C
    \left(\frac{t \sqrt{n}}{\epsilon} \right)^2 \left[\log 
      \left(\frac{t \sqrt{n}}{\epsilon} \right)\right] ^2 
\end{equation*}
for every $t > 0$ and $\epsilon > 0$. 
\end{lemma}

Lemma \ref{bmi} does not automatically imply Theorem \ref{sdb} simply
because the class $B(\B{0}, 1)$ is much larger than $B_{\infty}(\B{0},
n^{-1/2})$. Nevertheless, it turns out that the entries $\bt_{ij}$ of
a matrix $\bt$ in $B(\B{0}, 1)$ are bounded provided $\min(i-1, n_1 -
i)$ and $\min(j-1, n_2 - j)$ are not too small. This is the content of
Lemma \ref{isoineq} given below. 
\begin{lemma}{\label{isoineq}}
The following holds for every $\bt \in B(\B{0}, 1)$ and $1 \leq i \leq
n_1, 1 \leq j \leq n_2$:
\begin{equation}{\label{cft}}
|\bt_{ij}| \leq \max \left(\sqrt{\frac{1}{ij}}, \sqrt{\frac{1}{(n_1 -
         i + 1)(n_2 - j + 1)}} \right). 
\end{equation}
\end{lemma}
Using Lemma \ref{isoineq}, we employ a
peeling-type argument to prove Theorem \ref{sdb} where we partition
the entries of the matrix $\bt$ into various subrectangles and use
Lemma \ref{bmi} in each subrectangle. The complete proof of Theorem
\ref{sdb} along with the proofs of Lemma \ref{bmi} and Lemma
\ref{isoineq} are given in Appendix~\ref{appy}.  

\subsection{Proof of Theorem  \ref{kp1}} \label{compe} 
We provide the proof of Theorem \ref{kp1} here using the results from
the last two subsections. 

Fix $\bt^* \in \M$ and let $f_{\bt^*}(\cdot)$ be defined as in
\eqref{eq:chat} with $t_{\bt^*}$ being the point in $[0, \infty)$
where $t \mapsto f_{\bt^*}(t)$ attains its maximum.   

Let $\overline{\bt^*}$ denote the constant matrix taking the value
$\sum_{i=1}^{n_1} \sum_{j=1}^{n_2} \bt^*_{ij}/n$, i.e., 
$\overline{\bt^*}_{kl} = \sum_{i=1}^{n_1} \sum_{j=1}^{n_2}
\bt^*_{ij}/n$ for all $1 \leq k \leq n_1$ and $1 \leq l \leq
n_2$. Writing $\bt = \bt - \overline{\bt^*} + \overline{\bt^*}$, we
have 
\begin{equation*}
  \sup_{\bt \in B(\bt^*, t)} \left<\beps, \bt - \bt^* \right> =
  \sup_{\bt \in B(\bt^*, t)} \left<\beps, \bt - \overline{\bt^*} \right>
  + \left< \beps, \overline{\bt^*} - \bt^* \right>
\end{equation*}
for every $t \geq 0$. Taking expectations on both sides with respect
to $\beps$, we obtain 
\begin{equation}\label{ttr}
\E  \sup_{\bt \in B(\bt^*, t)} \left<\beps, \bt - \bt^* \right> =
 \E \sup_{\bt \in B(\bt^*, t)} \left<\beps, \bt - \overline{\bt^*} \right>.
\end{equation}
Now by the triangle inequality, it is easy to see that 
\begin{equation*}
  B(\bt^*, t) \subseteq B \left(\overline{\bt^*}, r_t \right)
  \qt{where $r_t := t + \sqrt{n V(\bt^*)}$}. 
\end{equation*}
This and \eqref{ttr} together imply that 
\begin{equation*}\label{gpin}
\E  \sup_{\bt \in B(\bt^*, t)} \left<\beps, \bt - \bt^* \right>  \leq \E \sup_{\bt \in B(\overline{\bt^*}, r_t)} \left<\beps, \bt - \overline{\bt^*} \right> .
\end{equation*}
Because $\overline{\bt^*}$ is a constant matrix, it is easy to see that 
\begin{equation*}
\sup_{\bt \in B(\overline{\bt^*}, r_t)} \left<\beps, \bt -
  \overline{\bt^*} \right> = \sup_{\bt \in B(\B{0},r_t)}  \left< \beps, \bt \right> = r_t \sup_{\bt \in B(\B{0}, 1)} \left< \beps, \bt \right>
\end{equation*}
where $\B{0}$ denotes the constant matrix with all entries equal to
$0$. 

As a consequence, we have
\begin{equation}\label{fft}
  f_{\bt^*}(t) \leq r_t ~ \E \sup_{\bt \in B(\B{0}, 1)} \left< \beps, \bt
  \right> - \frac{t^2}{2} \qt{for all $t \geq 0$}. 
\end{equation}
We now use Theorem~\ref{dudthm} with $\delta = 1/\sqrt{n}$ to obtain 
\begin{eqnarray*}
\E \sup_{\bt \in B(\B{0}, 1)} \left<\beps, \bt
  \right> \leq 12 \sigma \int_{1/\sqrt{n}}^{2} \sqrt{\log
    N(\epsilon, B(\B{0}, 1))} d\epsilon + 4 \sigma. 
\end{eqnarray*}
Inequality \eqref{cor1} with $\delta = n^{-1/2}$ then gives 
\begin{equation*}
  \E \sup_{\bt \in B(\B{0}, 1)} \left< \beps, \bt
  \right> \leq C \sigma \left( A (\log (B \sqrt{n}))^2 + 1 \right)
\end{equation*}
with $A := (\log n_1) (\log n_2)$ and $B := 4 \sqrt{(\log n_1) (\log n_2)}$.    

Thus, letting $g(t) := C r_t \sigma \left(A (\log (B \sqrt{n}))^2 + 1 
\right)$, we obtain from \eqref{fft} that 
\begin{equation*}
  f_{\bt^*}(t) \leq g(t) - \frac{t^2}{2} \qt{for all $t \geq 0$}. 
\end{equation*}
It can now be directly verified that 
\begin{equation*}
f_{\bt^*}(t^{**}) \leq  g(t^{**}) - \frac{1}{2} (t^{**})^2 \leq 0
\qt{for $t^{**} := 2 C \sqrt{\gamma^2 + \gamma (n V(\bt^*))^{1/2}} $}  
\end{equation*}
where $\gamma := \sigma \left( A (\log (B \sqrt{n}))^2 + 1
\right)$. Inequality~\eqref{cim} in Theorem~\ref{chatthm} therefore
gives 
\begin{equation}\label{ff}
  R(\hat{\bt}, \bt^*) \leq \frac{C}{n} \max \left((t^{**})^2, \sigma^2
  \right) . 
\end{equation}
Now $(t^{**})^2 = C(\gamma^2 + \gamma \sqrt{n V(\bt^*)})$ and using
the expressions for $A$ and $B$, it is easy to see that (note that $n
> 1$ because $n_1, n_2 > 1$)
\begin{equation*}
  \gamma = \sigma \left( A (\log (B \sqrt{n}))^2 + 1
\right) \leq C \sigma \left(\log n \right)^4. 
\end{equation*}
This, along with \eqref{ff}, allows us to deduce
\begin{equation*}
  R(\hat{\bt}, \bt^*) \leq C \left(\frac{\sigma^2}{n} (\log n)^8 + \sqrt{\frac{\sigma^2 V(\bt^*)}{n}} (\log n)^4 \right)
\end{equation*}
which proves Theorem \ref{kp1}. 

\section{Risk, Tangent Cones and the Proofs of Theorems \ref{kp2} and
  \ref{varadap}}{\label{2.3}}     
This section is devoted to the proofs of Theorem \ref{kp2} and Theorem
\ref{varadap}.  We use a recent result of \citet{bellec2015sharp} on the
connection between the risk $R(\bt^*,  \hat{\bt})$ and certain size
measures of tangent cones to $\M$ at $\bt^*$. This result is recalled
in the next subsection. 

\subsection{Risk and tangent cones} 
Fix $\bt \in \M$. The tangent cone of $\M$ at $\bt$ will be
denoted by $T_{\M}(\bt)$ and is defined as the closure of the convex
cone generated by $\B{u} - \bt$ as $\B{u}$ varies over $\M$ i.e., 
\begin{equation*}
  T_{\M}(\bt) := \text{closure} \left\{\alpha (\B{u} - \bt) :
    \alpha > 0 \text{ and } \B{u} \in \M \right\}. 
\end{equation*}
The tangent cone $T_{\M}(\bt)$ is a closed, convex subset of
$\R^n = \R^{n_1 \times n_2}$. Observe that if $\bt$ is a
constant matrix (i.e., all entries of $\bt$ are the same), then
$T_{\M}(\bt)$ is simply equal to $\M$. 

It turns out that the risk $R(\bt^*, \hat{\bt})$ can be controlled by
appropriate size measures of the tangent cones $T_{\M}(\bt), \bt \in
\M$. This is formalized in the following lemma. This lemma is similar
in spirit to results in \citet{OH13}. More general such results
involving model misspecification have recently appeared in
\citet{bellec2015sharp}. 

Let $\B{\epsilon} = (\beps_{ij})$ denote
the $n_1 \times n_2$ matrix all of whose entries are independent and
normally distributed with zero mean and variance 
$\sigma^2$. The Euclidean projection of $\B{\epsilon}$ onto the
tangent cone $T_{\M}(\bt)$ is defined in the usual way as   
\begin{equation*}
  \Pi(\B{\epsilon}, T_{\M}(\bt)) := \argmin_{\B{u} \in T_{\M}(\bt)}
  \|\B{\epsilon} - \B{u} \|^2. 
\end{equation*}
\begin{lemma}\label{oyha}
For every $\bt^* \in \M$ we have
\begin{equation}\label{reoo}
  R(\bt^*, \hat{\bt}) \leq \frac{1}{n} \inf_{\bt \in \M}
  \Big(\|\bt^*- \bt\|^2 + \E \| \Pi(\B{\epsilon}, T_{\M}(\bt)) \|^2
  \Big).  
\end{equation}
where the expectation on the right hand side is with respect to $\beps$. 
\end{lemma} 
\begin{proof}
Recall that $\B{y} = \bt^* + \B{\epsilon}$ and that $\hat{\bt}$ is the projection of the data matrix $\B{y}$ onto $\M$. By the usual KKT conditions, this projection
$\hat{\bt}$ satisfies 
\begin{equation*}\label{kk1}
  \left<\B{y} - \hat{\bt}, \hat{\bt} - \bt \right> \geq 0 \qt{for
    every $\bt \in \M$}
\end{equation*}
where $\left<A, B \right> = \sum_{i=1}^{n_1} \sum_{j=1}^{n_2} a_{ij} b_{ij}$ for $A = (a_{ij})$ and $B = (b_{ij})$. This inequality implies that 
\begin{equation*}
  \|\B{y} - \bt \|^2 \geq \|\B{y} - \hat{\bt}\|^2 + \|\hat{\bt} -
  \bt\|^2 \qt{for every $\bt \in \M$}.
\end{equation*}
Writing $\B{y} = \bt^* + \B{\epsilon}$, expanding out the squares and rearranging terms, we obtain 
\begin{eqnarray*}
\|\bt^* -\bt\|^2 + \|\B{\epsilon}\|^2 + 2 \langle \bt^* -\bt, \B{\epsilon} \rangle & \ge & \|\bt^* -\hat \bt\|^2 + \|\B{\epsilon}\|^2 + 2 \langle \bt^* -\hat \bt, \B{\epsilon} \rangle + \|\hat{\bt} - \bt\|^2 \\
\mbox{i.e.,} \quad  \|\hat{\bt} - \bt^*\|^2 & \leq & 2 \langle \hat \bt -\bt, \B{\epsilon} \rangle -  \|\hat{\bt} - \bt\|^2 + \|\bt^* -\bt\|^2 \\
\mbox{i.e.,} \quad \|\hat{\bt} - \bt^*\|^2 & \leq & \|\bt^* - \bt\|^2 + \|\B{\epsilon}\|^2 - \|\B{\epsilon} - (\hat{\bt} - \bt)\|^2. 
\end{eqnarray*}
Because $\hat{\bt} \in \M$, the matrix $\hat{\bt} - \bt$ belongs to
the tangent cone $T_{\M}(\bt)$. We therefore get
\begin{equation*}
  \|\hat{\bt} - \bt^*\|^2 \leq \|\bt^* - \bt\|^2 + \|\B{\epsilon}\|^2 - \inf_{\B{u} \in T_{\M}(\bt)} \|\B{\epsilon} - \B{u}\|^2.  
\end{equation*}
The infimum over $\B{u}$ above is clearly achieved for $\B{u} :=
\Pi(\B{\epsilon}, T_{\M}(\bt))$ and hence
\begin{equation}\label{bfh}
  \|\hat{\bt} - \bt^*\|^2 \leq \|\bt^* - \bt\|^2 + \|\B{\epsilon}\|^2
  - \|\B{\epsilon} - \Pi(\B{\epsilon}, T_{\M}(\bt))\|^2   
\end{equation}
Because $T_{\M}(\bt)$ is a closed convex cone, the projection
$\Pi(\B{\epsilon}, T_{\M}(\bt))$ satisfies (see, for example,
\cite[Equation (4)]{MW00}):
\begin{equation*}
  \left<\B{\epsilon} - \Pi(\B{\epsilon}, T_{\M}(\bt)),
    \Pi(\B{\epsilon}, T_{\M}(\bt)) \right> = 0. 
\end{equation*}
The above equality and inequality \eqref{bfh} together imply that
\begin{equation*}
  \|\hat{\bt} - \bt^*\|^2 \leq \|\bt^* - \bt\|^2 + \|\Pi(\B{\epsilon},
  T_{\M}(\bt))\|^2. 
\end{equation*}
The required inequality \eqref{reoo} now follows by taking
expectations on both sides. 
\end{proof}

Inequality \eqref{reoo} reduces the problem of bounding the risk to
controlling the expected squared norm of the projection of
$\B{\epsilon}$ onto the tangent cones $T_{\M}(\bt), \bt \in \M$.  This
will be crucially used in the proof of Theorem \ref{kp2}.  

\subsection{Proof of Theorem \ref{kp2}}\label{Sec4.2}
We provide the proof of Theorem \ref{kp2} in this subsection. The
first step is to characterize the tangent cone $T_{\M}(\bt)$ for every
$\bt \in \M$. We need some notation here. For a subset $S$ of $\{(i,
j): 1 \leq i \leq n_1, 1 \leq j \leq n_2\}$, let $\R^S$ denote the
class of all real-valued functions from $S$ to $\R$. Elements of
$\R^S$ will be denoted by $(\bt_{ij}, (i, j) \in S)$. We say that
$(\bt_{ij}: (i, j) \in S)$ is isotonic if  
\begin{equation*}
  \bt_{ij} \leq \bt_{kl} \qt{whenever $(i, j), (k, l) \in S$
    with $i \leq k$ and $j \leq l$}. 
\end{equation*}
The set of such isotonic sequences in $\R^S$ will be denoted by $\M(S)$. Also for every two dimensional array $\bt = (\bt_{ij} : 1 \leq i \leq n_1, 1 \leq j \leq n_2)$, let  
\begin{equation*}
  \bt(S) := (\bt_{ij} : (i, j) \in S). 
\end{equation*}
Observe that $\bt(S) \in \M(S)$ if $\bt \in \M$. The following lemma
provides a useful characterization of $T_{\M}(\bt)$ for $\bt \in
\M$. Recall that a rectangular partition of $[n_1] \times [n_2]$ is a
partition of $[n_1] \times [n_2]$ into rectangles. The cardinality
$|\pi|$ of a rectangular partition $\pi$ equals the number of
rectangles in the partition. The collection of all rectangular
partitions of $[n_1] \times [n_2]$ is denoted by $\pp$. We say that
$\bt \in \M$ is constant on $\pi = (A_1, \dots, A_k) \in \pp$ if
$\{\bt_{ij} : (i, j) \in A_l\}$  is a singleton for each $l$.  
\begin{lemma}\label{tcol}
Fix $\bt \in \M$  and $\pi = (A_1, \dots, A_k) \in \pp$ such that $\bt$
is constant on $\pi$. Then
\begin{equation}\label{tcol.eq}
    T_{\M}(\bt) \subseteq \left\{\B{v} \in \R^{n} : \B{v}(A_i) \in
      \M(A_i) \text{ for each } i = 1,\dots,k \right\}.  
\end{equation}
\end{lemma}
\begin{proof}
Suppose that $\B{v} = \alpha (\B{t} - \bt)$ for some $\B{t} \in \M$ and $\alpha >0$. This means that $\B{v}(A_i) = \alpha 
  (\B{t}(A_i) - \bt(A_i))$ for each $i$. Because $\B{t}(A_i) \in
  \M(A_i)$ and $\bt(A_i)$ is a constant ($\bt$ is constant on $\pi$), we now have $\B{v}(A_i) \in \M(A_i)$. As the right-hand side of~\eqref{tcol.eq} is a closed set, and $T_{\M}(\bt)$ is the closure of all such $\B{v}$'s, the desired result follows. 
\end{proof}

\begin{remark}
Note that we did not use the fact that $A_1, \dots, A_k$ are rectangular
in Lemma \ref{tcol}. We only used the fact that $\bt$ is constant on
each $A_i$. This means that \eqref{tcol.eq} is true also when $A_1,
\dots, A_k$ are the levels sets of $\bt$ i.e., each $A_l = \{(i, j) :
\bt_{ij} = a\}$ for some real number $a$. In fact, when $A_1, \dots,
A_k$  are the level sets of $\bt$, we have equality in
\eqref{tcol.eq}. This can be proved as follows. 

Suppose that $\B{v}(A_i) \in \M(A_i)$ for each $i$. We
shall argue then that $\bt + \alpha \B{v} \in \M$ for some $\alpha >
0$ which, of course, proves that $\B{v} \in T_{\M}(\bt)$. Observe
first that $A_1, \dots, A_k$ form a partition of $[n_1] \times
[n_2]$. Let $D$ denote the collection of all pairs $((i, j), (k, l))$
such that $i \leq j$ and $k \leq l$ and $\bt_{ij} \neq
\bt_{kl}$. Note, in particular, that $(i, j)$ and $(k, l)$ belong to
different elements of the partition $A_1, \dots, A_k$ if $((i,  j),
(k, l)) \in D$. Let  
\begin{equation*}
    \alpha := \min \left\{\frac{\bt_{kl} - \bt_{ij}}{\B{v}_{ij} -
        \B{v}_{kl}}:  ((i, j), (k, l))  \in D \text{ and } \B{v}_{ij}
      > \B{v}_{kl}\right\}.   
\end{equation*}
By monotonicity of $\bt$, it is clear that $\alpha > 0$. With this
choice of $\alpha$, it is elementary to check that  $\bt + \alpha
\B{v} \in \M$. This shows that \eqref{tcol.eq} is true with equality
when $A_1, \dots, A_k$ are the level sets of $\pi$. 
\end{remark}

We now have all the tools to complete the proof of Theorem \ref{kp2}. 
\begin{proof}[Proof of Theorem \ref{kp2}]
The first step is to observe via inequality \eqref{reoo} that it is
enough to prove the existence of a universal positive constant $C$ for
which  
  \begin{equation*}
    \E \|\Pi(\B{\epsilon}, T_{\M}(\bt))\|^2 \leq C k(\bt)
      \sigma^2 (\log n)^8 \qt{for all $\bt \in \M$}. 
  \end{equation*}
From the definition of $k(\bt)$, it is enough of prove that 
\begin{equation}
  \label{fofo.1}
  \E \|\Pi(\B{\epsilon}, T_{\M}(\bt))\|^2 \leq C k \sigma^2 (\log n)^8
\end{equation}
for every $\pi = (A_1, \dots, A_k) \in \pp$ such that $\bt$ is
constant on $\pi$. To prove \eqref{fofo.1}, use the
characterization of $T_{\M}(\bt)$ in Lemma \ref{tcol} to observe that  
\begin{equation}\label{fof}
   \E \|\Pi(\B{\epsilon}, T_{\M}(\bt))\|^2 \leq \sum_{i=1}^k
   \E \|\Pi(\B{\epsilon}(A_i), \M(A_i))\|^2 . 
\end{equation}
The task then reduces to that of bounding $\E \|\Pi(\B{\epsilon}(A_i),
\M(A_i))\|^2 $   for $i = 1, \dots, k$. It is crucial that each $A_1,
\dots, A_k$ is a rectangle. Fix 
$1 \leq i \leq k$ and without loss of generality assume that $A_i =
[n_1'] \times [n_2']$ for some $1 \leq n_1' \leq n_1$ and $1 \leq n_2'
\leq n_2$. It is then easy to see that Theorem \ref{kp1} for $\bt^* =
\B{0}$ and $n_1 = n_1'$,  $n_2 = n_2'$ immediately gives  
\begin{equation}\label{aniy}
    \E \|\Pi(\B{\epsilon}(A_i), \M(A_i))\|^2 \leq C \sigma^2 (\log
    (2 n_1'n_2'))^8 
\end{equation}
for a universal positive constant $C$ as long as $n_1' > 1$ and $n_2'
> 1$.  When $n_1' = n_2' = 1$, it can be checked that the left hand
side of \eqref{aniy} equals $\sigma^2$ which means that \eqref{aniy}
is still true provided $C$ is changed accordingly. Finally when
$\min(n_1', n_2') = 1$ and $\max(n_1', n_2') > 1$, one can use the
result \eqref{motw} from vector isotonic estimation to prove
\eqref{aniy}. We thus have
\begin{equation*}
  \E \|\Pi(\B{\epsilon}(A_i), \M(A_i))\|^2 \leq C \sigma^2 (\log n)^8 
\end{equation*}
for a universal constant $C$ for all $n_1' \geq 1$ and $n_2' \geq
1$. This inequality together with inequality \eqref{fof} implies
\eqref{fofo.1} which completes the proof of Theorem \ref{kp2}.    
\end{proof}

\subsection{Proof of Theorem \ref{varadap}}\label{Sec4.3}
We now give the proof of Theorem \ref{varadap}. Let us first prove
inequality \eqref{vara1}. For $\theta \in \C_{n_1}$,
let $\Upsilon(\theta) \in \M$ be defined by $\Upsilon(\theta)_{ij} =
\theta_i$ for all $1 \leq i \leq n_1$ and $1 \leq j \leq n_2$. Also
let $\Upsilon(\C_{n_1}) := \left\{\Upsilon(\theta): \theta \in
  \C_{n_1} \right\}$. Note first that all level sets of
$\Upsilon(\theta)$ are rectangular for every $\theta \in \C_{n_1}$
which implies that  
$k(\Upsilon(\theta)) = k(\theta)$ for every $\theta \in
\C_{n_1}$. Therefore, as a consequence of Theorem \ref{kp2}, we obtain
that for every $\bt^* \in \M$, 
\begin{equation*}
  R(\bt^*, \hat{\bt}) \leq  \inf_{\theta \in \C_{n_1}} \left(\frac{\|\bt^*
      - \Upsilon(\theta)\|^2}{n}  + \frac{C k(\theta) \sigma^2}{n} (\log
    n)^8 \right). 
\end{equation*}
Now if there exists $\theta^* \in \C_{n_1}$ such that $\Upsilon(\theta^*) = \bt^*$, then it is obvious that $\|\bt^* - \Upsilon(\theta) \|^2 = n_2
\|\theta^* - \theta\|^2$ which proves \eqref{vara1}. 

Inequality \eqref{vara2} can now be derived from \eqref{vara1} by a
standard approximation argument. For every $\theta^* \in \C_{n_1}$ with $D =
D(\theta^*) = (\theta^*_{n_1} - \theta^*_1)^2$ and $0 \leq \delta \leq
\sqrt{D}$, there exists $\theta \in \C_{n_1}$ with 
\begin{equation*}
  \frac{\|\theta - \theta^*\|^2}{n_1} \leq \delta^2 ~~ \text{  and  }
  ~~ k(\theta) \leq \frac{2 \sqrt{D}}{\delta}. 
\end{equation*}
This fact is easy to prove and a proof can be found, for example, in
\citet[Lemma B.1]{us}. Using this, it follows directly from
\eqref{vara1} that 
\begin{equation*}
  R(\bt^*, \hat{\bt}) \leq C \inf_{0 < \delta \leq \sqrt{D}}
  \left(\delta^2 + \frac{2 \sigma^2 \sqrt{D}}{n \delta} (\log n)^8
  \right). 
\end{equation*}
The choice $\delta = (2 \sigma^2 \sqrt{D}/n)^{1/3}$ now leads to
inequality \eqref{vara2}. This choice of $\delta$ satisfies $\delta
\leq \sqrt{D}$ provided $n D \geq 2 \sigma^2$. This completes the
proof of Theorem \ref{varadap}. 

\section{Discussion}\label{cuss}
In this paper we have studied the risk behavior of the LSE of an
unknown matrix $\bt^*$, constrained to be nondecreasing in both rows
and columns, when observed with errors. We prove both worst case and
adaptive risk bounds for the LSE. A highlight of the adaptation
properties of the LSE is that it adapts automatically to the intrinsic
dimension of the problem. 

Two further research questions are mentioned below.  

The logarithmic factors in our risk bounds, e.g., in
Theorems~\ref{kp1} and~\ref{kp2}, are probably not optimal. They arise as a consequence of (i) the presence of logarithmic factors in the covering number result in~\citet{GW07} (see the proof of Lemma~\ref{bmi}), and (ii) the fact that the entropy integral in~\eqref{cor1} in Theorem~\ref{sdb} diverges to $+\infty$ if $\delta
\downarrow 0$. It is not clear to us at the moment how to remove or reduce these logarithmic factors. 

In this paper we deal with the estimation of an isotonic matrix. It is natural to ask how the results generalize to isotonic tensors of higher order, and more generally to estimating a multivariate isotonic regression function under general designs. It would be interesting to
see whether such adaptation results hold in these situations.

\section*{Acknowledgements} The second author would like to thank Sivaraman Balakrishnan for helpful discussions.
\appendix
\section{Appendix}\label{appy}
\subsection{Proof of Lemma~\ref{bmi}}
For each $\bt \in B_{\infty}(\B{0}, t)$, we associate a bivariate
coordinate-wise nondecreasing function $\phi_{\bt} : [0,1]^2
\rightarrow \R$ via 
\begin{equation*}
  \phi_{\bt}(x_1, x_2) := \min \left\{\bt_{ij} : n_1 x_1 \leq i \leq 
    n_1, n_2 x_2 \leq j \leq n_2 \right\} 
\end{equation*}
for all $(x_1, x_2) \in [0, 1]^2$. It can then be directly verified
that  
\begin{equation*}
\|\bt - \B{\nu}\|^2  = n \int_0^1 \int_{0}^1
(\phi_{\bt}(x_1,x_2) -  \phi_{\B{\nu}}(x_1,x_2))^2 dx_1dx_2  
\end{equation*}
for every pair $\bt,\B{\nu} \in B_{\infty}(\B{0}, t)$. Moreover, if $\C([0,1]^2, t)$ denotes the class of all bivariate coordinate-wise
nondecreasing functions that are uniformly bounded by $t$, then it is straightforward to verify that $\phi_{\bt} \in \C([0,1]^2,
t)$ for every $\bt \in B_{\infty}(\B{0}, t)$. These two latter facts
immediately imply that 
\begin{equation}\label{gte}
N(\epsilon, B_{\infty}(\B{0}, t)) \leq  N\left(n^{-1/2} \epsilon/2,
  \C([0, 1]^2, t), L_2 \right) 
\end{equation}
where $N(\epsilon/2, \C([0, 1]^2, t), L_2)$ denotes the
$\epsilon/2$-covering number of $\C([0, 1]^2, t)$ under the $L_2$
metric $L_2(f, g) := \left(\int (f - g)^2 \right)^{1/2}$. This latter
covering number has been studied by \citet{GW07} who proved that  
\begin{equation*}
  N(\epsilon/2, \C([0, 1]^2, t), L_2) \leq C \left(\frac{t}{\epsilon}
  \right)^2 \left[\log \left(\frac{t}{\epsilon} \right) \right]^2
\end{equation*}
for a universal positive constant $C$. This and~\eqref{gte} together
complete the proof of Lemma \ref{bmi}.   \qed

\subsection{Proof of Lemma~\ref{isoineq}}
Fix $\bt \in B(\B{0}, 1)$ and $1 \leq i \leq n_1, 1 \leq j \leq
n_2$. Our proof of \eqref{cft} involves considering the following two
cases separately:  
\begin{enumerate}
\item $\bt_{ij} < 0$: Here, by monotonicity of $\bt$, the inequality
  $\bt_{kl} \le \bt_{ij}$ must hold for all $1 \leq k \leq i$ and $1
  \leq l \leq j$. Therefore, $|\bt_{kl}| \geq 
|\bt_{ij}|$ holds  for all $(k,l) \in [1,i]\times [1,j].$ Finally
because $\bt \in B(\B{0}, 1)$, we have 
\begin{equation*}
   1 \geq  \sum_{k = 1}^{i}\sum_{l = 1}^{j} \bt_{kl}^2
   \geq ij\bt_{ij}^2. 
\end{equation*}
This proves \eqref{cft} when $\bt_{ij} < 0$. 
\item $\bt_{ij} \geq 0.$ Here by monotonicity of $\bt$, the condition
  $\bt_{kl} \geq \bt_{ij}$ must hold for all $i \leq k \leq n_1$ and
  $j \leq l \leq n_2.$ Therefore, by nonnegativity 
of $\bt_{ij}$ and by virtue of $\bt \in B(\B{0}, 1)$ we have  
\begin{equation*}
   1 \geq  \sum_{k = i}^{n_1}\sum_{l = j}^{n_2} \bt_{kl}^2
   \geq (n_1 + 1 - i)(n_2 + 1 - j) \bt_{ij}^2. 
\end{equation*}
This proves \eqref{cft} when $\bt_{ij} \geq 0$. 
\end{enumerate} \qed

\subsection{Proof of Theorem~\ref{sdb}}
The basic idea behind this proof is the following. By Lemma
\ref{isoineq}, it is clear that for every matrix $\bt \in B(\B{0},
1)$, the entries $\bt_{ij}$ are bounded by constants provided $\min(i
- 1, n_1 - i)$ and $\min(j - 1, n_2 - j)$ are not too small. Further,
for bounded isotonic matrices, the metric entropy bounds can be
obtained from Lemma \ref{bmi}. We shall therefore employ a
peeling-type argument where we partition the entries of $\bt$ into
various subrectangles and use Lemma \ref{bmi} in each subrectangle. 

Let us introduce some notation. Let $B$ denote the set
$B(\B{0}, 1)$ for simplicity.  For a subset $S \subset 
\{(i, j) : 1 \leq i \leq n_1, 1 \leq j \leq n_2\}$ with cardinality
$|S|$ and $\bt \in \M$, let $\bt(S) \in \R^{|S|}$ be defined as 
\begin{equation*}
  \bt(S) := (\bt_{ij} : (i, j) \in S)). 
\end{equation*}
Further let $B_S$ denote the collection of all $\bt(S)$ as $\bt$ ranges
over $B$. The $\epsilon$-metric entropy of $B_S$ (under the Euclidean
metric on $\R^{|S|}$) will be denoted by $N(\epsilon, B_S)$. 

We first prove inequality \eqref{sdb.eq}. Let $I_1 := \{i: 1 \leq i
\leq n_1/2 \}$ and $I_2 := \{i : n_1/2 < i   
\leq n_1 \}$. Also $J_1 := \{j: 1 \leq j \leq n_2/2 \}$ and $J_2 :=
\{j: n_2/2 < j \leq n_2 \}$. Because 
\begin{equation*}
\|\bt - \B{\alpha}\|^2 = \sum_{1
  \leq k, l \leq 2} \|\bt(I_k \times J_l) - \B{\alpha}(I_k \times
J_l)\|^2  
\end{equation*}
for all $\bt$ and $\B{\alpha}, $ it follows that 
\begin{equation*}
 \log N(\epsilon, B) \leq  \sum_{k=1}^2 \sum_{l=1}^2
 \log N(\epsilon/2,B_{I_k \times J_l}). 
\end{equation*}
We shall prove below that for every $1 \leq k, l \leq 2$ and $\epsilon
> 0$, 
\begin{equation}\label{pd}
  \log N(\epsilon/2,B_{I_k \times J_l}) \leq C \frac{(\log n_1)^2 (\log
    n_2)^2}{\epsilon^2} \left[\log \frac{4 \sqrt{\log n_1 \log
        n_2}}{\epsilon} \right]^2
\end{equation}
for a universal positive constant $C$. This would then complete the
proof of \eqref{sdb.eq}. 

Let $k_1$ and $k_2$ denote the smallest integers for which $2^{k_1} >
n_1/2$ and $2^{k_2} > n_2/2$. For every $0 \leq u < k_1$ and $0 \leq v
< k_2$, let 
\begin{align*}
&N^{1}_u := \{i \in I_1: 2^u \leq i \leq \min(2^{u+1} - 1, n_1/2) \} ~
  \text{and} ~\\  &N^{1}_v := \{j \in J_1: 2^v \leq j \leq
  \min(2^{v+1} - 1, n_2/2) \}.  
\end{align*}
Similarly let 
\begin{align*}
&N^{2}_u := \{i \in I_2: 2^u \leq n_1 + 1 - i \leq \min(2^{u+1} - 1,
n_1/2) \} ~ \text{and} ~\\ &N^{2}_v := \{j \in J_2: 2^v \leq n_2 + 1 -
j \leq \min(2^{v+1} - 1, n_2/2) \}. 
\end{align*}
For each pair $1 \leq k, l \leq 2$, because 
\begin{equation*}
\|\bt(I_k \times J_l) - \B{\alpha}(I_k \times J_l)\|^2 =
\sum_{u=0}^{k_1 - 1} \sum_{v = 0}^{k_2 - 1} \|\bt(N_u^k \times
  N_v^l) - \B{\alpha}(N_u^k \times N_v^l) \|^2
\end{equation*}
it follows that 
\begin{equation}\label{nra}
  \log N(\epsilon/2, B_{I_k \times J_l}) \leq \sum_{u=0}^{k_1 - 1}
  \sum_{v=0}^{k_2 - 1} \log N(k_1^{-1/2} k_2^{-1/2}\epsilon/2,
  B_{N_u^k \times N_v^l}). 
\end{equation}
Now fix $0 \leq u < k_1, 0 \leq v < k_2$ and $1 \leq k, l  \leq 2$. We
argue below that $N(k_1^{-1/2} k_2^{-1/2}\epsilon/2, B_{N_u^k \times
  N_v^l})$ can be controlled using Lemmas~\ref{bmi} and~\ref{isoineq}. Note first that the cardinality of $N_u^k \times N_v^l$
is at most $|N_u^k| |N_v^l| \leq 2^{u+v}$. We also claim that
\begin{equation}{\label{claim}}
  \max_{i \in N^{k}_u, j \in N^{l}_v} |\bt_{ij}| \leq 2^{-(u+v)/2}
    \qt{for all $\bt \in B$}. 
\end{equation}
We will prove the above claim a little later. Assuming for now that it
is true, we can use Lemma~\ref{bmi} for $B_{N_u^k \times N_v^l}$ to
deduce that 
\begin{equation*}
\log N(k_1^{-1/2} k_2^{-1/2} \epsilon/2, B_{N_u^k \times N_v^l}) \leq
C \frac{k_1k_2}{\epsilon^2} \left( \log \frac{4k_1^{1/2}
    k_2^{1/2}}{\epsilon} \right)^2
\end{equation*}
for a universal positive constant $C$. Inequality~\eqref{nra} then
gives 
\begin{equation}\label{nag}
  \log N(\epsilon/2,B_{I_k \times J_l}) \leq C
  \frac{k^2_1k^2_2}{\epsilon^2} \left( \log \frac{4k_1^{1/2}
      k_2^{1/2}}{\epsilon} \right)^2.  
\end{equation}
Because $k_1$ is the smallest integer for which $2^{k_1} > n_1/2$, we
have $2^{k_1 - 1} \leq n_1/2$ which means that $k_1 \leq \log
n_1$. Similarly $k_2 \leq \log n_2$. This together with~\eqref{nag}
implies~\eqref{pd} which completes the proof of \eqref{sdb.eq}. The
only thing that remains now is to prove~\eqref{claim}.  

We first prove \eqref{claim} for $k = l = 1.$ By Lemma \ref{isoineq}, 
we get that $|\bt_{ij}| \leq (ij)^{-1/2}$ for all  for $\bt \in B$ and
$(i,j) \in I_1 \times J_1$. Clearly $\min_{i \in N^1_u} i = 2^u$ and
$\min_{j \in N^1_v} i = 2^v$. This proves \eqref{claim} for $k = l =
1$. A similar argument will also work for $k = l =2$. For the case
when $k = 1, l = 2$, note that   
\begin{equation*}
\max_{N^{1}_u \times N^{2}_v} \bt_{ij} \leq \max_{N^{2}_u \times
  N^{2}_v} \bt_{ij} \leq 2^{-(u + v)/2}
\end{equation*}
which follows from the monotonicity of $\bt$ and \eqref{claim} for $k
= l = 2$. Similarly, 
\begin{equation*}
\min_{N^{1}_u \times N^{2}_v} \bt_{ij} \geq \min_{N^{1}_u \times
  N^{1}_v} \bt_{ij} \geq -2^{-(u + v)/2}. 
\end{equation*}
Putting these together, we obtain \eqref{claim} for $k = 1, l = 2.$ A
similar argument will work for $k = 2, l = 1$. This completes the
proof of \eqref{sdb.eq}.  

For \eqref{cor1}, simply observe that by \eqref{sdb.eq}, 
\begin{align*}
\int_{\delta}^{1} \sqrt{\log N(\epsilon, B(\B{0}, 1))} \; d\epsilon
&\le 
\sqrt{C}\sqrt{A}  \int_{\delta}^{1} \frac{1}{\epsilon} \left(\log
  \frac{B}{\epsilon}\right) \; d\epsilon \\& 
= \frac{\sqrt{C}\sqrt{A}}{2} \left[ (\log \frac{B}{\delta})^2 - (\log
  B)^2 \right] \le \frac{\sqrt{C}\sqrt{A}}{2} \left (\log \frac{B}{\delta}
\right)^2. 
\end{align*}
This completes the proof of Theorem \ref{sdb}.  \qed

\subsection{Proof of Theorem \ref{lobo}}\label{App:A4}
We shall use Assouad's lemma to prove Theorem \ref{lobo}. The
following version of Assouad's Lemma is a consequence of Lemma 24.3
of~\citet[pp. 347]{van2000asymptotic}.   
\begin{lemma}[Assouad]\label{suad}
Fix $D > 0$ and a positive integer $d$. Suppose that, for each $\tau \in \{-1, 1\}^d$, there is an associated $\B{g^{\tau}}$ in $\M$ with $D(\B{g^{\tau}}) \leq D$. Then 
  \begin{equation*}
\inf_{\tilde{\bt}} \sup_{\bt \in \M: D(\bt) \leq D} R(\bt,
\tilde{\bt}) \geq \frac{d}{8} \min_{\tau \neq \tau^{'}}
    \frac{\ell^2(\B{g}^{\tau}, \B{g}^{\tau'})}{\ham(\tau,\tau^{'})} \min_{\ham (\tau, \tau^{'}) = 1} \left(1 - \|\P_{\B{g}^{\tau}} - \P_{\B{g}^{\tau'}}\|_{TV} \right), 
  \end{equation*}
where $\ham(\tau, \tau^{'}) := \sum_{i=1}^d I\{\tau_i \neq
\tau^{'}_i\}$ denotes the Hamming distance between $\tau$ and
$\tau^{'}$ and $\|\cdot \|_{TV}$ denotes the total variation
distance. The notation $\P_{\B{g}}$ for $\B{g} \in \M$ refers to the
joint distribution of $\by_{ij} = \B{g}_{ij} + \beps_{ij}$, for $i \in
[n_1], j \in [n_2]$ when $(\beps_{ij})$ are independent normally
distributed random variables with mean zero and variance $\sigma^2$.  
\end{lemma}

We are now ready to prove Theorem \ref{lobo}. 

Fix $D > 0$ and an integer $k$ with $1 \leq k \leq \min(n_1,
n_2)$. Let $m_1$ and $m_2$ be defined so that $k = \lfloor n_1/m_1
\rfloor = \lfloor n_2/m_2 \rfloor$. Let $d = k^2$. We denote elements
of $\{-1, 1\}^d$ by 
$(\tau_{uv}: u, v \in [k] \times [k])$. For each such $\tau \in \{-1,
1\}^{d}$, we define $\B{g}^{\tau} \in \M$ in the following way. For $i
\in [n_1]$ and $j \in [n_2]$, if there exist $u, v \in [k]$ for which
$(u-1)m_1 < i \leq u m_1$ and $(v-1)m_2 < j \leq v m_2$, we take  
\begin{equation*}
  \B{g}^{\tau}_{ij} = \sqrt{D} \left( \frac{u + v - 2}{2k} + \frac{\tau_{uv}}{6k}
    \right). 
\end{equation*}
Otherwise we take $\B{g}^{\tau}_{ij} = \sqrt{D}$. One can check that $\B{g}^{\tau} \in \M$ and $D(\B{g}^{\tau}) \leq D$ for every $\tau \in \{-1,1\}^d$. 

We shall now use Lemma~\ref{suad} with $d = k^2$ and this collection $\{\B{g}^{\tau} : \tau \in \{-1, 1\}^d\}$. Note first that  
\begin{align*}
\ell^2(\B{g}^{\tau}, \B{g}^{\tau'}) &= \frac{1}{n} \sum_{i=1}^{n_1} \sum_{j=1}^{n_2} \left(\B{g}^{\tau}_{ij} - \B{g}^{\tau'}_{ij} \right)^2 \\ 
&= \frac{1}{n} \sum_{u, v \in [k]} \sum_{i : (u-1)m_1 < i \leq u m_1} \sum_{j: (v-1) m_2 < j \leq v m_2} \left(\B{g}^{\tau}_{ij} - \B{g}^{\tau'}_{ij} \right)^2 \\
&= \frac{D}{n} \sum_{u, v \in [k]} \frac{m_1 m_2}{36 k^2}
\left(\tau_{uv} - \tau'_{uv} \right)^2 = \frac{m_1 m_2 D}{9 n k^2}
\ham(\tau, \tau').   
\end{align*}
Therefore, this implies that 
\begin{equation*}
\min_{\tau \neq \tau'} \frac{\ell^2(\B{g}^{\tau}, \B{g}^{\tau'})}{\ham(\tau,
  \tau')} = \frac{m_1 m_2 D}{9 n k^2}. 
\end{equation*}
To bound $\|\P_{\B{g}^{\tau}} - \P_{\B{g}^{\tau'}}\|_{TV}$, we use Pinsker's inequality because the Kullback-Leibler divergence $D(\P_{\B{g}^{\tau}} \|
\P_{\B{g}^{\tau'}})$ has a simple expression in terms of $\ell^2(\B{g}^{\tau},
\B{g}^{\tau'})$:  
\begin{equation*}
\|\P_{\B{g}^{\tau}} - \P_{\B{g}^{\tau'}}\|^2_{TV} \leq \frac{1}{2} D \left(\P_{\B{g}^{\tau}} \| \P_{\B{g}^{\tau'}} \right) = \frac{n}{4 \sigma^2} \ell^2 \left(\B{g}^{\tau}, \B{g}^{\tau'} \right)= \frac{m_1 m_2 D}{36 \sigma^2 k^2} \ham(\tau, \tau').  
\end{equation*}
This gives
\begin{equation*}
 \min_{\ham (\tau, \tau^{'}) = 1} \left(1 - \|\P_{\B{g^{\tau}}} -
      \P_{\B{g}^{\tau'}}\|_{TV} \right) \geq \left(1 - \frac{\sqrt{m_1 m_2 D}}{6 k \sigma} \right).
\end{equation*}
Lemma \ref{suad} then gives the lower bound for $\Delta :=
\inf_{\tilde{\bt}} \sup_{\bt \in \M: D(\bt) \leq D} R(\bt,
\tilde{\bt})$ as given below:
\begin{equation*}
\Delta \geq \frac{m_1 m_2 D}{72 n} \left(1 - \frac{\sqrt{m_1 m_2D}}{6
    k \sigma} \right). 
\end{equation*}
Because $k = \lfloor n_i/m_i \rfloor$ for $i = 1, 2$, it follows that $n_i/(k+1) \leq m_i \leq n_i/k$ for $i = 1, 2$. This gives
\begin{equation*}
\Delta \geq \frac{D}{72(k+1)^2} \left(1 - \frac{ \sqrt{n D}}{6 \sigma
    k^2} \right) \geq \frac{D}{288 k^2} \left(1 - \frac{\sqrt{nD}}{6
    \sigma k^2} \right) 
\end{equation*}
where we have also used that $k+1 \le 2k$. The choice $k = (nD/(9
\sigma^2))^{1/4}$ now leads to $\Delta \geq \sigma \sqrt{D}/(192
\sqrt{n})$. This gives what we wanted to prove provided our choice of
$k$ satisfies $1 \leq k \leq \min(n_1, n_2)$. For this, it suffices to
simply note that $n \geq 9 \sigma^2/D$ implies that $k \geq 1$ and
\eqref{lobo.con} implies $k \leq \min(n_1, n_2)$. This completes the
proof of Theorem \ref{lobo}. \qed

\bibliographystyle{plainnat}
\bibliography{AG}
\end{document}